\documentclass[12pt,reqno]{amsart}
\usepackage{amsfonts,amssymb,amsmath,amsopn,amsthm,graphicx}
\usepackage{amsxtra, mathrsfs}
\usepackage{colordvi}
\usepackage[usenames,dvipsnames]{color}
\usepackage{amsfonts,amssymb,amsbsy,amsmath,amsthm,dsfont}

\usepackage{amsthm}
\usepackage{amsbsy}
\usepackage{amstext}
\usepackage{amsfonts}
\usepackage{mathcomp}
\usepackage{mathtools}
\usepackage{dsfont}
\usepackage{tabularx}
\usepackage{latexsym}
\usepackage{amssymb}
\usepackage{esint}

\usepackage{lipsum}
\usepackage[many]{tcolorbox}

\usepackage[colorlinks,pdfpagelabels,pdfstartview = FitH,bookmarksopen = true,bookmarksnumbered = true,linkcolor = blue,
plainpages = false,hypertexnames = false,citecolor = red,pagebackref=false]{hyperref}
\usepackage{cite}

\usetikzlibrary{decorations.pathreplacing}

 \numberwithin{equation}{section}

\newtcolorbox{leftbrace}{%
	enhanced jigsaw, 
	breakable, 
	frame hidden, 
	overlay={%
		\draw [
		decoration={brace,amplitude=0.5em},
		decorate,
		ultra thick,
		]
		(frame.south west)--(frame.north west);
	},
	parbox=false,
}

\usepackage{amssymb,amsmath,graphicx,amsthm,a4wide,wrapfig,caption,subcaption,epstopdf}
\usepackage[latin1]{inputenc}
\usepackage{enumitem}

\usepackage{verbatim}

\usepackage{lineno}
\usepackage{textcomp}
\usepackage{mathtools,hyperref}
\usepackage{cleveref}

\usepackage{setspace,esint}

\newcommand{\Hmm}[1]{\leavevmode{\marginpar{\tiny%
			$\hbox to 0mm{\hspace*{-0.5mm}$\leftarrow$\hss}%
			\vcenter{\vrule depth 0.1mm height 0.1mm width \the\marginparwidth}%
			\hbox to
			0mm{\hss$\rightarrow$\hspace*{-0.5mm}}$\\\relax\raggedright #1}}}

\newcommand{\bb}{{\mathbf{\bar {b}}}}
\newcommand{\bt}{{\mathbf{\tilde{b}}}}
\newcommand{\Gwb}{{\overline{\Omega}}}

\newcommand{\Pw}{\partial \Omega}

\newcommand{\loc}{{\rm loc}}

\newcommand{\N}{\mathbb{N}}

\newcommand{\R}{\mathbb{R}}

\newcommand{\Dir}{\mathrm{Dir}}
\newcommand{\pwd}{\partial\Omega_{\mathrm{Dir}}}
\newcommand{\pwr}{\partial\Omega_{\mathrm{Rob}}}
\newcommand{\wb}{\overline{\Omega}\setminus \partial\Omega_{\mathrm{Dir}} }

\newtheorem{theorem}{Theorem}[section]

\newtheorem{cor}[theorem]{Corollary}
\newtheorem{thm}[theorem]{Theorem}
\newtheorem{lem}[theorem]{Lemma}
\newtheorem{lemma}[theorem]{Lemma}
\newtheorem{proposition}[theorem]{Proposition}
\newtheorem{example}[theorem]{Example}

\newtheorem{definition}[theorem]{Definition}
\newtheorem{defi}[theorem]{Definition}

\newtheorem{rem}[theorem]{Remark}

\theoremstyle{definition}

\newtheorem{assumptions}[theorem]{Assumptions}

\numberwithin{equation}{section}
\newcommand{\diver}{\mathrm{div}\,}
\newcommand{\RN}[1]{%
	\textup{\uppercase\expandafter{\romannumeral#1}}%
}
\newcommand{\dx}{\,\mathrm{d}x}
\newcommand{\dy}{\,\mathrm{d}y}

\newcommand{\dt}{\,\mathrm{d}t}

\newcommand{\ds}{\,\mathrm{d}s}

\newcommand{\dsigma}{\,\mathrm{d}\sigma}

\newcommand{\core}{C_0^{\infty}(\Omega)}

\newcommand{\be}{\begin{equation}}
\newcommand{\ee}{\end{equation}}
\newcommand{\bea}{\begin{eqnarray}}
\newcommand{\eea}{\end{eqnarray}}
\newcommand{\bean}{\begin{eqnarray*}}
	\newcommand{\eean}{\end{eqnarray*}}

\newcommand{\opname}[1]{\mbox{\rm #1}\,}
\newcommand{\supp}{\opname{supp}}

\newlength{\wex}  \newlength{\hex}
\newcommand{\ass}[1]{Let Assumptions~\ref{assump1} hold  in a bounded Lipschitz domain $\Gw$}

\def\ga{\alpha}     \def\gb{\beta}       \def\gg{\gamma}

       \def\vgf{\varphi}    \def\gh{\eta}
            \def\gl{\lambda}

\def\Gw{\Omega}              

\begin{document}

	\title[Hardy weights for mixed boundary value problems]{Optimal Hardy-weights for elliptic operators with mixed
		boundary conditions}
	
	
	
%

	\author{Yehuda Pinchover}
	
	\address{Yehuda Pinchover,
		Department of Mathematics, Technion - Israel Institute of
		Technology,   Haifa, Israel}
	
	\email{pincho@technion.ac.il}
	
	\author {Idan Versano}
	
	\address {Idan Versano, Department of Mathematics, Technion - Israel Institute of
		Technology,   Haifa, Israel}
	
	\email {idanv@campus.technion.ac.il}
	\begin{abstract}
		We construct families of optimal Hardy-weights for a subcritical linear second-order elliptic operator $(P,B)$ with degenerate mixed boundary conditions.  By an optimal Hardy-weight for a subcritical operator we mean a nonzero nonnegative weight function $W$ such that $(P-W,B)$ is critical, and null-critical with respect to $W$. Our results rely on
		a recently developed criticality theory for  positive solutions of the corresponding mixed boundary value problem.
		
		\medskip
		
		\noindent  2000  \! {\em Mathematics  Subject  Classification.}
			Primary  \! 35A23; Secondary  35B09, 35J08, 35J25.\\[1mm]
	\end{abstract}
	\maketitle

  
		%
%
	\section{Introduction}

The classical  Hardy inequality states that if $n\geq 3$, then
\begin{equation}\label{eq:classical_Hardy}
\int_{\R^n\setminus\{0\}}|\nabla \phi|^2\dx\geq
\left ( \frac{n-2}{2}\right)^2\int_{\R^n\setminus \{0\}}
\frac{|\phi|^2}{|x|^2}\dx \qquad \forall \phi\in C_0^{\infty}(\R^n\setminus \{0\}). 
\end{equation}
It is well known that the constant $( (n-2)/2)^2$ in \eqref{eq:classical_Hardy} is sharp but not achieved in the appropriate Hilbert space.
In the past four decades,  
Hardy-type inequalities were developed for  vast classes of operators with different types of boundary conditions, see for example, \cite{DFP,KPP,FT,KL,DP,EHL} and references therein.
In \cite{DFP}  the authors established a general method to generate an explicit
optimal Dirichlet-Hardy-weight for a general (not necessarily symmetric) subcritical linear elliptic
operators either in divergence form or in non-divergence form defined on a domain $\Gw\subset \R^n$ or on a noncompact manifold. In some definite
sense, an optimal Hardy-weight is `as large as possible' Hardy-weight.

In the present paper, we utilize the approach developed in \cite{DFP,PV} to produce  families of optimal Hardy-inequalities 
 for a general linear, second-order, elliptic operator with degenerate mixed boundary conditions. Our approach relies on  
 criticality theory of positive weak solutions for elliptic operators with mixed boundary conditions which has been recently established in \cite{PV}. Though our methods are similar to those in \cite{DFP,PV}, where (generalized) Dirichlet boundary conditions are considered, a more intricate treatment is required when mixed boundary conditions are considered. 
 
 We remark that in \cite{KL} (resp.,\cite{EHL}) a variational approach was taken in order to produce Hardy-inequalities  in a  Lipschitz (resp., $C^2$-convex) domain   for the Laplace (resp., $p$-Laplace)  operator with degenerate  Robin boundary conditions on $\pwd$. 
  Our approach does not assume convexity of the domain,  and provides, in particular,  a nontrivial improvement of
 a certain (Robin) Hardy-inequality  obtained  in \cite{KL} for the Laplacian in the exterior of the unit ball (see Example~\ref{example:2}). See also \cite{EHR} for related results.
 
 We recall some essential definitions and results  which 
 are discussed in detail in \cite{PV2}.
 Let  $P$ be  a second-order, linear, elliptic operator  with real measurable coefficients which is  defined on a domain $\Gw\subset \R^n$. We assume that $P$ is  in the divergence form  
 \begin{equation} \label{div_P}
 Pu:=-\diver \left[A(x)\nabla u +  u\bt(x) \right]  +
 \bb(x) \cdot \nabla u   +c(x)u \qquad x\in\Gw.
 \end{equation}
Let $\pwr$ be a relatively open $C^1$-portion of $\partial \Gw$, and consider the oblique boundary operator
 \begin{equation}\label{B_weak_for}
 Bu:=\beta(x) \big(A(x) \nabla u+u\bt(x)\big)\cdot \vec{n}(x) +\gamma(x) u \qquad x\in \pwr,
 \end{equation}
 where $\vec{n}(x)$ is the outward unit normal vector to $\partial \Gw$ at $x\in \Pw_{\mathrm{Rob}}$,  $\xi\cdot\eta$ denotes the Euclidean inner product of $\xi, \eta\in \R^n$, and $\beta,\gamma$ are  real measurable functions defined on $\Pw_{\mathrm{Rob}}$ such that $\gb>0$ on $\pwr$.
 The boundary  of $\Gw$ is then naturally decomposed to a disjoint union of its Robin part $\Pw_{\mathrm{Rob}}$, and its Dirichlet part $\Pw_{\mathrm{Dir}}$. That is,  $\partial\Omega=\Pw_{\mathrm{Rob}}\cup \Pw_{\mathrm{Dir}}$, where  $\Pw_{\mathrm{Rob}}\cap \Pw_{\mathrm{Dir}}=\emptyset$.
 If further $\bb=\bt$ in $\Gw$, we say that $(P,B)$ is {\em symmetric} in $\Gw$.
 \begin{defi}\label{def:HW}
 	\em{
 	Let $\Gw$ be a Lipschitz domain in $\R^n$ and consider the  operator $(P,B)$ defined on $\wb$,  where $P$ and $B$ are of the forms \eqref{div_P} and \eqref{B_weak_for}, respectively.
 	\begin{itemize}
 		\item We say that $(P,B)$ is {\em nonnegative in $\Gw$ } (in short $(P,B)\geq 0$ in $\Gw$) if there exists  a  positive weak  solution to the mixed boundary value problem
 		\begin{equation}\label{p,b}
 		\begin{cases}
 		Pu=0 &  \text{in } \Omega, \\\tag{{\bf P,B}}
 		Bu = 0 &  \text{on } \Pw_{\mathrm{Rob}}.
 		\end{cases}
 		\end{equation} 
 		\item We say that $W \gneqq  0$ is a {\em Hardy-weight} for $(P,B)$ in $\Gw$ if 
 		$(P-W,B)\geq 0$ in $\Gw$.
 		\item A Hardy-weight  $W$ is said to be {\em a Dirichlet-Hardy-weight} if $\pwr=\emptyset$.
 		\item A nonnegative operator $(P,B)$ in $\Gw$ 
 		is said to be {\em subcritical} (resp., {\em critical}) in $\Gw$  		
 		if $(P,B)$ admits (resp., does not admit) a Hardy-weight for $(P,B)$ in $\Gw$. 
 	\end{itemize} 
}
 \end{defi}

By \cite[Theorem 5.19]{PV2},  $(P,B)$ is subcritical in $\Gw$ if and only if $(P,B)$ admits a positive minimal Green function $G_{P,B}^{\Gw}(x,y)$ in $\Gw$. On the other hand,  $(P,B)$ is critical in $\Gw$ if and only if the equation $(P,B)u=0$ admits (up to a multiplicative constant) a unique positive supersolution $\phi$ in $\Gw$. In fact, $\phi$  is  a positive solution, called the {\em (Agmon) ground state} (for the definition of a ground state see Definition~\ref{def_gs}). Furthermore, $(P,B)$ is critical in $\Gw$ if and only if $(P^\star,B^*)$ is critical in $\Gw$, where $(P^\star,B^*)$ is the formal adjoint of $(P,B)$ in $L^2(\Gw)$ \cite[Corollary 5.20]{PV2}.

Next we recall the notion of optimal Hardy-weight and optimality at $\infty$ \cite{DFP}.
\begin{definition}\label{def_opt_Hardy}
	{\em  Let $W$ be a Hardy-weight for $(P,B)$ in $\Gw$. A Hardy-weight $W$ is said to be {\em optimal in $\Gw$} if $(P-W,B)$ is critical in $\Gw$ and $\int_{\Gw}\phi \phi^*W\dx=\infty$, where $\phi$ and $\phi^*$ are respectively, the ground states of $(P-W,B)$ and $(P^\star-W,B^*)$ in $\Gw$.  In this case we say that $(P-W,B)$ is {\em null-critical } with respect to the weight $W$. 			
	}
\end{definition}
\begin{defi}\label{rem_1.3}{\em 
		We say that a Hardy-weight $W$ is {\em optimal at infinity in $\Gw$} if for any   $K\Subset \overline{\Gw}$,
		$\partial K \cap \pwd =\emptyset$, and $\partial K \cap \pwr \Subset \pwr$ with respect to the relative topology on $\pwr$ (in short, $K\Subset_R \Gw$), we have 
		$$\sup\{\gl \in \R \mid (P-\gl W,B)\geq 0 \mbox{ in } \Gw\setminus K\}=1.$$   	
	}
\end{defi}
\begin{rem}{\em
 The definition of an optimal Hardy-weight in \cite{DFP} includes the requirement that $W$ should be optimal at infinity in $\Gw$. But, for the case $\pwr=\emptyset$, it is proved  in \cite{KP}  that if $P - W$ is null-critical with respect to $W$ in $\Gw$, then $W$ is also optimal at infinity. The same proof applies to the more general case considered in the present paper, hence, in Definition~\ref{def_opt_Hardy} we avoid the requirement of optimality at infinity.  
}
\end{rem}	
 \begin{defi}\label{def_Gp}
	\em{
		 Let $(P,B)$ be a subcritical operator in $\Gw$, and let  $G(x,y):=G^{\Gw}_{P,B}(x,y)$ the corresponding positive  minimal Green function. Fix $0 \lneqq \vgf\in\core$.
		The {\em Green potential with a density} $\vgf$ is the function 
		$$
		G_{\vgf}(x):=\int_{\Gw}G(x,y)\vgf(y)\dy.
		$$  	
	}
\end{defi}
Next, we define the notion of a function vanishing near the Dirichlet boundary part and at $\infty$. 
 We first need to define an exhaustion of $\Gw$ subordinated to $\wb$.
\begin{defi}\label{def_exhaus}
	\em{
		A sequence $\{\Omega_k\}_{k\in \mathbb{N}} \subset \Omega$ 
		is called an {\em exhaustion}  of $\wb$ if it is
		an increasing sequence of Lipschitz domains such that $\Omega_k\Subset_R \Omega_{k+1}\Subset_R  \Gw$, and $\bigcup\limits_{k\in \N} \overline{\Omega_k}=\overline{\Gw}\setminus \Pw_{\mathrm{Dir}}$. 
		For $k\geq 1$ we denote : 
		$$  \partial \Gw_{k,\mathrm{Rob}}:=\mathrm{int}(\partial \Gw_k\cap \pwr) \quad 
		\partial \Gw_{k,\Dir}:=\partial \Gw_k\setminus \partial \Gw_{k,\mathrm{Rob}}, \quad 
		\Omega_{k}^{*}:=(\wb)\setminus \overline{\Omega_k}. $$
	}
\end{defi}
\begin{rem}\label{rem:boundness}
	\em{
			For the existence of such an exhaustion see the appendix  in \cite{PV2}.	
			We note that the construction of $\{\Gw_k\}_{k\in \N}$ ensures that  $\partial\Gw^{*}_{k,\mathrm{Rob}}$ is Lipschitz as well.
	}
\end{rem}
\begin{defi}
	\em{
		Let $K\Subset \Gw$ and $f\in C((\overline{\Gw\setminus K})\setminus \partial \Gw_{\mathrm{Dir}})$. 
		We say that 
		\begin{equation}\label{eq:infty_bar}
		\lim_{x\to \infty_{\Dir}}f(x)=0
		\end{equation}
		if for any $\varepsilon>0$ and any exhaustion $\{ \Gw_k\}_{k\in \N}$ of $\wb$, there exists $k_0$ such that 
		$|f(x)|<\varepsilon$ in $\Gw\setminus\Gw_{k_0}$.
	}
\end{defi}
\begin{rem}\label{rem:boundness2}
	\em{
		\begin{enumerate}
			\item 	In the case $\pwr\!=\!\emptyset$, \eqref{eq:infty_bar} is equivalent to the condition 
			$\lim\limits_{x\to \overline{\infty}}f(x)=0$, where $\overline{\infty}$ denotes the ideal point in the one-point compactification of $\Gw$.
			\item Let $K\Subset \Gw$ , then  $f\in C((\overline{\Gw\setminus K})\setminus \pwd)$ satisfying \eqref{eq:infty_bar} is bounded in $\Gw\setminus K$.
		\end{enumerate}
	}
\end{rem}
The first main result of the paper provides an explicit family of optimal Hardy weights. It  reads as follows:
\begin{thm}\label{main_thm}
	Let Assumptions~\ref{assump2} hold in a domain $\Omega \subset \mathbb{R}^n$, ($n \geq 2$) containing $x_0$. Let $(P,B)$ be a subcritical   operator  in $\Omega$ and let $G(x):=G_{P,B}^{\Omega}(x,x_0)$ be its  minimal positive Green function with singularity at $x_0\in \Gw$.  Let  $G_{\varphi}$ be the Green potential with  a density $0\lneqq \varphi \in \core$. Assume that there exists $u\in \mathcal{H}^{0}_{P,B}(\Omega)$  satisfying Ancona's condition
	\begin{equation}\label{Anc_cond}
	\lim_{x \to \infty_{\Dir}} \dfrac{G(x)}{u(x)}=0.
	\end{equation}
	Let $$0\leq  a \leq \dfrac{1}{\sup\limits_{\Gw}{(G_{\varphi} /u)}}\, ,\quad
	f_w(t):=\sqrt{2t-at^2}, \quad \mbox{and }\;\;
	W:=\frac{P \left (uf_w \left ( G_{\varphi} /u\right) \right )}{uf_w \left ( G_{\varphi} /u\right)}\,.
	$$
	Then $(P-W,B)$ is critical  in $\Gw$, and
	\begin{equation}\label{eq:W_explicitly}
	W=\frac{|\nabla(G_{\varphi}/u)|_{A}^2}{(G_{\varphi} /u)^2(2-a G_{\varphi}/u)^2}\qquad  \text{in }  \Gw \setminus \supp(\varphi).
	\end{equation}
	Furthermore, assume that  one of the following regularity conditions  are  satisfied.
	\begin{enumerate}
			\item  $(P,B)$ is symmetric,  $A\in C^{0,1}_{\loc}(\wb, \R^{n^2})$, $\bb=\bt\in C^{\alpha}_{\loc}(\wb,\R^n)$,  $c\in L^{\infty}_{\loc}(\wb)$, and $\pwr\in C^{1,\alpha}$.
				
		\item $\pwr$, $\pwd$ are both relatively open and closed  sets, $\pwr$ is bounded and admits a finite number of connected components, and the coefficients of $P$ are smooth functions in $\Gw$ (or more generally, $A\in C^{\lceil(3n-2)/2\rceil,1}_{\loc}(\Gw,\R^{n^2})$, $\bt\in C^{\lceil(3n-2)/2\rceil,1}_{\loc}(\Gw,\R^{n})$, $\bb\in C^{\lceil(3n-2)/2\rceil-1,1}_{\loc}(\Gw,\R^{n})$, $c\in C^{\lceil(3n-2)/2\rceil-1,1}_{\loc}(\Gw)$). 
\end{enumerate}
	Then  $W$ is an optimal Hardy-weights for $(P,B)$ in $\Gw$.
\end{thm}
\begin{rem}
	\em{	
		Theorem \ref{main_thm} is proved   in \cite{DFP} for the specific (classical) case
		$\pwr=\emptyset$ and $a=0$. Moreover, Theorem \ref{main_thm} with $\pwr=\emptyset$ and $a>0$ was proved in \cite{PV}, and provides greater Hardy-weights than the optimal Hardy-weight $W_{\mathrm{class}}$ given in \cite{DFP}, in the sense that
		$$
		W=\dfrac{|\nabla (G_{\vgf}/u)|_A^2}{(G_{\vgf}/u)^2\big (2-a(G_{\vgf/u}) \big )^2}>W_{\mathrm{class}}:=\dfrac{|\nabla(G_\varphi /u)|_{A}^2}{4(G_{\varphi} /u)^2} \qquad \mbox{in } \Gw \setminus (\supp(\varphi)\cup \mathrm{Crit}(G_{\vgf}/u)).
		$$		
Clearly, the above inequality holds true also in the case $\pwr \neq\emptyset$.
	}
\end{rem}
The second main result of the paper is  a generalization of Theorem~\ref{main_thm}. 
We show that  any optimal Dirichlet-Hardy-weights for  the one-dimensional Laplacian in $\R_+$  provides an $n$-dimensional optimal Hardy weights for $(P,B)$ in $\Gw$ (cf. \cite[Theorem 5.2]{PV}). 
\begin{defi}{\em 
	We say that $0\lneqq w\in L^1_{\loc}(\R_+)$ belongs to $\mathcal{W}(\R_+)$ if $w$ is an optimal Dirichlet-Hardy-weight of $Ly:=-y''$ in $\R_{+}$.
}
\end{defi}
For a characterization of optimal Dirichlet-Hardy-weights for Sturm-Liouville operators see \cite[Proposition~3.1]{PV} (see also Proposition~\ref{1D if and only if}). 
\begin{theorem}\label{criticality_linear} \label{thm3}
	Assume that  the operator $(P,B)$, and the functions $G$, $u$, and $G_\vgf$ satisfy the assumptions of Theorem~\ref{main_thm}. 
	
	Let  $w \in \mathcal{W}(\R_+)$, and let $\psi_w(t)$ be the corresponding  ground state. 
	Suppose further that $\psi_w'\geq 0$ on  $\{t=G_{\varphi}(x)/u(x) \mid x\in \Gw \}$, and set
	\begin{equation*}
	W:=\dfrac{P \left (u \psi_w \left ( G_{\varphi} /u\right) \right )}{u \psi_w \left ( G_{\varphi} /u\right)}\,.
	\end{equation*}
	Then, the following assertions are satisfied:
	\begin{enumerate}
	\item	$W\geq 0$ in $\Gw$ and $W:=\left |\nabla (G_{\varphi}/u) \right |_{A}^{2}w(G_{\varphi}/u), ~ \text{in} ~ \Gw \setminus \supp(\varphi).$
	
	\item 			$(P-W,B)$ is critical  in $\Omega$ with ground state $u \psi_w (G_{\varphi}/u)$.
	
	\item Assume that $\pwr$, $\pwd$ are both relatively open and closed  sets, $\pwr$ is bounded and admits a finite number of connected components, and the coefficients of $P$ are smooth functions in $\Gw$ (or, $A\in C^{\lceil(3n-2)/2\rceil,1}_{\loc}(\Gw,\R^{n^2})$, $\bt\in C^{\lceil(3n-2)/2\rceil,1}_{\loc}(\Gw,\R^{n})$, $\bb\!\in \! C^{\lceil(3n-2)/2\rceil-1,1}_{\loc}(\Gw,\R^{n})$, $c \!\in \! C^{\lceil(3n-2)/2\rceil-1,1}_{\loc}(\Gw)$). Then $W$ is optimal at infinity.
	
	\item If $(P,B)$ is symmetric, $A\!\in \! C^{0,1}_{\loc}(\wb)$, $\bb=\bt\in C^{\alpha}(\wb,\R^n)$, and $c\in L^{\infty}_{\loc}(\wb)$, then $(P-W,B)$ is null-critical with respect to $W$,  and therefore, $W$ is an optimal Hardy-weight for $(P,B)$ in $\Gw$.	
\end{enumerate}
\end{theorem}
The paper is organized as follows. In Section~\ref{sec_prelim}, we introduce the necessary notation and recall some previously obtained results needed in the present paper. We proceed in Section~\ref{sec:optimal}, with proving a  Khas'minski\u{i}-type criterion for the mixed boundary value problem and prove our main results. Finally, in Section \ref{sec:exmaples}, we present a couple of examples.
 \section{Preliminaries and  notation}\label{sec_prelim}
Let $\Omega$ be a domain in $\R^n$, $n\geq 2$, and let $\partial\Omega=\Pw_{\mathrm{Rob}}\cup \Pw_{\mathrm{Dir}}$, where  $\Pw_{\mathrm{Rob}}\cap \Pw_{\mathrm{Dir}}=\emptyset$. We assume that $\Pw_{\mathrm{Rob}}$, the Robin portion of $\partial\Omega$, is a relatively open subset of $\partial \Gw$, and $\Pw_{\mathrm{Dir}}$, the Dirichlet part of $\Pw$, is a closed set of $\partial \Gw$. Moreover, if $\Gw$ is a bounded domain, we further assume that in the relative topology of $\partial \Gw$ we have $\mathrm{int}( \Pw_{\mathrm{Dir}})\neq\emptyset$.
Throughout the paper we use the following notation and conventions:
\begin{itemize}
	
	\item For any $\xi\in \R^n$ and a positive definite symmetric matrix $A\in \R^{n \times n}$, let  
	$| \xi|_A:=\sqrt{A\xi\cdot \xi  }$, where $\xi\cdot\gh$, and also $\langle \xi , \eta \rangle$, denote the Euclidean inner product of $\xi,\gh\in \R^n$.
	
	\item From time to time we use the Einstein summation convention.
	
	\item Let $g_1,g_2$ be two positive functions defined in $\Gw$. We use the notation $g_1\asymp g_2$ in
	$\Gw$ if there exists a positive constant $C$ such
	that
	$$C^{-1}g_{2}(x)\leq g_{1}(x) \leq Cg_{2}(x) \qquad \mbox{ for all } x\in \Gw.$$	
	
	\item Let $g_1,g_2$ be two positive functions defined in $\Gw$, and let $\xi\in \overline\Gw \cup \{\infty\}$. We use the notation $g_1\sim g_2$ near $x_0$ if there exists a positive constant $C$ such that
	$$
	\lim_{x\to \xi}\frac{g_1(x)}{g_2(x)}=C.
	$$	
	\item The gradient of a function $f$ will be denoted either by $\nabla f$ or $Df$.
		\item $\pwd$ and $\pwr$ denote the Dirichlet and Robin parts of $\partial \Gw$, respectively.
	\item  We write $A_1 \Subset A_2$ if  $\overline{A_1}$ is a compact set, and $\overline{A_1}\subset A_2$.
	\item For a subdomain  $\Gw' \! \subset \!\Gw$, let $\partial\Omega'_{\mathrm{Rob}}\! :=\!\mathrm{int}(\partial \Omega' \!\cap \! \pwr)$, and 
	$\partial \Omega'_{\mathrm{Dir}} \!:=\! \partial \Omega' \!\setminus\!  \partial\Omega'_{\mathrm{Rob}}$. 	
	\item Let $\Gw'$ be a subdomain of $\Gw$. We write $\Gw'\Subset_R \Gw$ if $\Gw'\Subset \overline{\Gw}$,
	$\partial \Gw' \cap \pwd =\emptyset$.

	\item $C$ refers to a positive constant which may vary from  line to line.
	\item For any $x=(x_1,\ldots,x_n)\in \R^n$, we use the notation $x=(x',x_n)$.
	\item For any $R>0$ and $x\in \R^n$, $B_R(x)$ is the open ball of radius $R$ centered at $x$.
	
	
	\item For any real measurable function $u$ and a measurable set $\omega\subset \R^n$, 
	$$\inf_{\omega}u:=\mathrm{ess}\inf_{\omega}u, \quad \sup_{\omega}u:=\mathrm{ess}\sup_{\omega}u, \quad  u^+:=\max(0,u), \quad u^-:=\max(0,-u).$$    
	\item For any real function $W$, we say that $W\gneqq 0$ in $\Gw$ if $W\geq 0$ in $\Gw$ and $\sup\limits_{\Gw}W>0$.
	\item  $\mathcal{H}^l$, $1\leq l \leq n$, denotes the $l$-dimensional Hausdorff measure on $\R^n$.
	\item For a differentiable function $f$, $\mathrm{Crit}(f)$ denotes the set of critical points of $f$. 
\end{itemize}
Let $R,K>0$, and let $w$ be a real-valued Lipschitz continuous function defined on $\{x'\in \R^{n-1}:|x'|<R\}$ with 
$$
|w(x')-w(y')|\leq K|x'-y'|
$$
for all $x',y'\in B(x_0',R)$, that satisfy  $w(0)=0$. We define
\begin{align*}				
\Gw[R]=\{x\in \R^n: x_n>w(x'), \; |x|<R \}, \\ 
\sigma[R]=\{x\in \R^n: x_n \geq w(x'), \;|x|=R \}, \\ 
\Sigma[R]=\{x\in \R^n: |x|<R, \; x_n=w(x') \}.
\end{align*}
\begin{defi}[Lipschitz and $C^1$ portions]
	{\em 
		Let $x_0\in \Pw$ and $R>0$ such that $\Gw[x_0,R]:=\Gw\cap B_R(x_0)$ is a Lipschitz (resp.,  $C^1$) domain. The set 
		$\Sigma[x_0,R]=\Pw \cap  B_R(x_0)$ is called a {\em Lipschitz} (resp.,  $C^1$) {\em portion} of $\Pw$.
	} 
\end{defi}
Further, we introduce some functional spaces. 
Let $E\subset \partial \Gw$ be a closed subset in the relative topology of $\partial \Gw$. We define
$\mathcal{D}(\Omega,E):=  C_0^{\infty}(\overline{\Omega} \setminus E)$. So,  $u\in\mathcal{D}(\Omega,E)$ if $u$ has compact support, and 
$\supp u:=\overline{\{x\in \Gw\mid u(x)\neq 0\}} \subset \overline{\Omega} \setminus E$.

For $q \geq 1$, we define $ W^{1,q}_{E}(\Omega)$ to be the closure of 
$\mathcal{D}(\Omega,E)$  with respect to the Sobolev norm of $W^{1,q}(\Omega)$.
We also consider the following spaces: 
$$
L^{q}_{\loc}(\wb):=\{ u \mid  \forall x\in \Gw\cup \Pw_{\mathrm{Rob}},  \;   \exists r_x>0 \mbox{ s.t. } u\in L^q(\Omega\cap B_{r_x}(x))   \},$$	
$$
W^{1,q}_{\loc}(\overline{\Omega}\setminus \pwd):=\{ u\mid \forall x\in \Gw\cup \Pw_{\mathrm{Rob}},  \;   \exists r_x>0 \mbox{ s.t. } u\in W^{1,q}(\Omega\cap B_{r_x}(x))   \}.
$$
If $q=2$ we  write 
$H^{1}_{\loc}(\overline{\Omega}\setminus \pwd):=W^{1,2}_{\loc}(\overline{\Omega}\setminus \pwd)$.
\begin{rem}
	\em{
		For every Lipschitz subdomain $\Gw'\Subset_R \Gw$ and $1< q<\infty$
		the space $W^{1,q}(\Gw')$ is a reflexive Banach space and therefore,  $W^{1,q}_{\Pw_{\mathrm{Dir}}}(\Omega')$
		is reflexive as well. 
	}
\end{rem}
The space $W^{1,q}_{E}(\Gw)$, can be characterized using the following result. 
\begin{lem}[{\cite[Theorem 2.1]{ET},\cite[Remark 2.1]{T}}]\label{lem:alhors_2}
	Let $\Gw$ be a bounded  Lipschitz domain. Let $E\!\subset \! \partial \Gw$ be a compact subset with respect to the  relative topology on $\partial \Gw$, and  let $\mathrm{d}_{E}(x)\!:=\!\mathrm{dist}(x,E)$, where $x \!\in\! \Gw$. 
	Then, for $u\in H^1(\Gw)$ the following assertions are
	equivalent:  
	\begin{itemize}
		\item $\;\; u\in H^1_{E}(\Gw)$.\\
		\item $\;\; \displaystyle{\int_{\Gw}}\left |\frac{u}{\mathrm{d}
			_{E}}\right|^2 \dx<\infty.$\\
		\item  $\;\; \displaystyle{\lim\limits_{r \to 0}\frac{1}{|B_{r}(x)|}\int_{B_r(x)\cap \Gw}}|u|\dx=0\,$ for $\mathcal{H}^{n-1}$-almost every $x\in E$.
	\end{itemize}
\end{lem}
\begin{cor}\label{rem:ahlors} 
	Let $\Gw\subset \R^n$ be a bounded  Lipschitz domain.
	\begin{enumerate}
		\item 
		Assume that $v$  and $ u$  belong to $H^1(\Gw)$, and $u\geq 0$. Assume further that $v\in H^1_{\pwd}(\Gw)$. Then, $(u-v)^{-}\leq |v|$ in $\Gw$  and  $(u-v)^{-}\in H^{1}_{\pwd}(\Gw)$. 
		\item Let  $u,v\in H^1(\Gw)\cap C(\overline{\Gw})$, and $u\leq v$ on $\pwd$. Then $(v-u)^{-}\in H^1_{\pwd}(\Gw)$.
	\end{enumerate}
\end{cor}
\begin{proof}
	Let $E=\pwd$. 
\begin{enumerate}
	\item Lemma \ref{lem:alhors_2} implies that 
	$$v\in H^{1}_{\pwd}\Longrightarrow \int_{\Gw}\left |\frac{v}{\mathrm{d}_E}\right |^2 \dx<\infty.$$
	The inequality $(u-v)^{-}\leq |v|$ is trivial when $u\geq 0$, and therefore,
	$$\int_{\Gw}\left |\frac{(u-v)^{-}}{\mathrm{d}_E}\right |^2 \dx\leq \int_{\Gw}\left |\frac{v}{\mathrm{d}_E}\right |^2 \dx<\infty.$$
		\item The continuity of $u$ and $v$ implies that $(v-u)^{-}$  vanishes continuously on $E$. In particular,
	$\;\; \displaystyle{\lim\limits_{r \to 0}\frac{1}{|B_{r}(x)|}\int_{B_r(x)\cap \Gw}}|(v-u)^{-}|\dx=0\,$ for  every $x\in E$. \qedhere
\end{enumerate}
\end{proof}
Consider an elliptic operator  $P$ of the form \eqref{div_P} and a Robin boundary operator $B$ of the form \eqref{B_weak_for}. 
{\bf Throughout the paper we assume  the following regularity assumptions on $P$, $B$ and $\partial \Gw$:}
\begin{leftbrace}\begin{assumptions}
		\label{assump2} 
		\begin{itemize}
			\item[{\ }]		
			\item $\partial\Omega=\Pw_{\mathrm{Rob}}\cup \Pw_{\mathrm{Dir}}$, where  $\Pw_{\mathrm{Rob}}\cap \Pw_{\mathrm{Dir}}=\emptyset$, and $\Pw_{\mathrm{Rob}}$ is a relatively open subset of $\Pw$.
			\item For each $x_0\in \Pw_{\mathrm{Rob}}$,   $\exists R>0$ such that $\Sigma[x_0,R]$ is a $C^1$-portion of $\partial\Gw$.
			\item $A\!=\!(a^{ij})_{i,j=1}^{n}\in L_\loc^\infty(\Gwb \setminus \Pw_{\mathrm{Dir}}; \R^{n\times n})$ is a symmetric positive definite  matrix valued function which is 
			locally uniformly elliptic  in $\Gwb \setminus \Pw_{\mathrm{Dir}}$, that is, for any compact $K\subset \Gwb \setminus \Pw_{\mathrm{Dir}}$ there exists  $\Theta_K>0$ such that 
			\begin{eqnarray*} 
				\Theta_K^{-1}\sum_{i=1}^n\xi_i^2\leq\sum_{i,j=1}^n
				a^{ij}(x)\xi_i\xi_j\leq \Theta_K\sum_{i=1}^n\xi_i^2 \quad \forall \xi\in \mathbb{R}^n \mbox{ and } \forall x\in K.
			\end{eqnarray*}	
			\item $\bt, \bb\in  L^{p}_{\loc}(\wb;\R^n)$, and $c\in L^{p/2}_{\loc}(\wb)$ for some $p>n$.
			
			\item $\beta>0$, and $\gamma/\beta\in L^{\infty}_{\mathrm{loc}}(\Pw_{\mathrm{Rob}})$.	
		\end{itemize}
	\end{assumptions}
\end{leftbrace}
%

	In the sequel, we  use the following terminology.
\begin{defi}
	\em{
		Let $u\in H^1_{\loc}(\wb)$.
		We say that $u\geq 0 $ on $\pwd$ if $u^{-}\in H^1_{\pwd}(\Gw)$.
	}
\end{defi}
Next, we define (weak) solutions and supersolutions of the mixed boundary value problem
\begin{equation}\label{P,B}
\begin{cases}
Pu=0 &  \text{in } \Omega, \\\tag{{\bf P,B}}
Bu = 0 &  \text{on } \Pw_{\mathrm{Rob}}.
\end{cases}
\end{equation}
\begin{defi}{\em
		We say that $u\in H^{1}_{\loc}(\wb)$ is a  {\em weak solution (resp., supersolution)} of \eqref{P,B} in $\Gw$, if for any (resp., nonnegative) $\phi \in \mathcal{D}(\Omega,\Pw_{\mathrm{Dir}})$   we have
		\begin{equation*}\label{weak_solution}
		\mathcal B_{P,B}(u, \phi):=	\int_{\Omega}\big[(a^{ij}D_ju+u\bt^i )D_i \phi+(\bb^i D_i u+cu )\phi \big] \mathrm{d}x + \int_{\Pw_{\mathrm{Rob}}}\frac{\gamma}{\beta}u \phi \,  \mathrm{d}\sigma  = 0\;
		(\text{resp.,} \geq 0),
		\end{equation*}
		where $\mathrm{d}\sigma$ is the $(n-1)$-dimensional surface  measure. In this case we write $(P,B)u\!=\!0$ (resp., $(P,B)u\!\geq \!0$). Also, $u$ is a subsolution of \eqref{P,B}  if $-u$ is a supersolution of \eqref{P,B}.}	
\end{defi}
The above definition should be compared with the following standard definition of weak (super)solutions of the equation $Pu=0$ in a domain $\Gw$.
\begin{defi}
	\em{
		We say that $u\in H^{1}_{\mathrm{loc}}(\Gw)$ is a \em{weak solution (resp., supersolution)} of the equation $Pu=0$ in $\Gw$ if for any (resp., nonnegative) $\phi\in C_0^{\infty}(\Gw)$ 
		$$ \mathcal B_{P}(u, \phi):=
		\int_{\Omega}[(a^{ij}D_ju+u\bt^i )D_i \phi+(\bb^i D_i u+cu) \phi] \mathrm{d}x  = 0\;
		(\text{resp.,} \geq 0).
		$$
	}
\end{defi}
Hence, any weak solution (resp., supersolution) of the equation $(P,B)u=0$ in $\Gw$ is a weak solution (resp., supersolution)  of $Pu=0$ in $\Gw$. 
In the sequel, by a (super)solution of \eqref{P,B} we always mean a weak (super)solution.  
\begin{rem}
	\em{
	By \cite[Section 3.2]{PV2}, any positive solution of the problem $(P,B)v=0$ in $\Gw$ is H\"older continuous in $\wb$. Furthermore, $u>0$ on $\pwr$.
	}
\end{rem}
The {\em formal $L^2$-adjoint} of the operator $(P,B)$ is given by the operator $(P^*,B^*)$
\begin{equation}\label{P*,B*}
\begin{cases}
P^*u:=-\diver \left[A \nabla u+\bb u \right]+\bt\cdot \nabla u +cu,\\[2mm]\tag{{\bf P*,B*}}
B^*u:=\beta \big( A \nabla u+u\bb\big)\cdot \vec{n}+\gamma u.
\end{cases}
\end{equation}
Indeed, if $\bt,\bt$ and $\partial \Gw$ are sufficiently smooth, then 
for any $\phi,\psi \in \mathcal{D}(\Omega,\Pw_{\mathrm{Dir}})$  satisfying  $B\psi=B^* \phi =0$ on $\Pw_{\mathrm{Rob}}$ in the classical sense, we have
\begin{align*}
&  \int_{\Omega} P(\psi) \phi \dx=
\int_{\Omega} \big(-\diver(A\nabla \psi+\bt\psi )+\bb\cdot \nabla \psi+c\psi\big)\phi\dx = 
\\
& \int_{\Pw_{\mathrm{Rob}}} \psi \phi\big(\frac{\gamma}\beta+  \bb\cdot \vec{n}\big)  \dsigma+
\int_{\Omega}\big((A\nabla \psi+\bt\psi )\cdot \nabla \phi
-\psi \nabla\cdot (\bb \phi)+c\psi\phi \big)\dx=
\\
&
\int_{\Omega} \psi\big(\!-\diver(A\nabla \phi+\bb\phi )+\bt\cdot \nabla \phi+c\phi\big)\!\dx=
\int_{\Omega}\psi P^* (\phi) \dx.
\end{align*}
Next, we define the notion of nonnegativity of the operator $(P,B)$.
\begin{defi}
	\em{
		We denote the  cone  of all positive solutions and positive supersolutions of the equation $(P,B)u=0$ in $\Gw$ by $ {\mathcal H}^0_{P,B}(\Omega)$ and $\mathcal{SH}_{P,B}(\Omega)$, respectively.
		The operator  $(P,B)$ is said to be {\em nonnegative in $\Omega$} (in short, 
		$(P,B)\geq 0$ in $\Gw$) if $\mathcal{H}^{0}_{P}(\Omega)\neq \emptyset$.
	}
\end{defi}
The Fredholm alternative\cite[Section~2]{PV2} implies:
\begin{lemma} \label{Lem_Fred}
	Assume that $(P,B)\geq 0$ in $\Gw$ and let $\Gw'\Subset_R \Gw$ be a Lipschitz subdomain of $\Gw$.
	For any $\Upsilon=g_0+\diver\mathbf{g} \in (H^1_{\partial \Gw'_{\mathrm{Dir}}}(\Gw'))^*$ there exists a unique solution $u_{\Upsilon}\in H^1_{\partial \Gw'_\mathrm{Dir}}(\Gw')$ to the problem
	\begin{equation}\label{eq:Rob_func}
	\begin{cases}
	Pw=g_0+\diver\mathbf{g} & \quad \text{in } \Gw', \\
	Bw= - \mathbf{g}\cdot \vec{n} & \quad \text{on } \partial \Gw'_{\mathrm{Rob}},
	\end{cases}
	\end{equation}
	(in short, $(P,B)w=g_0+\diver \mathbf{g}$ in $\Gw'$) in the following (weak) sense:  
	
	For all $\phi\in \mathcal{D}(\Gw',\partial \Gw'_{\mathrm{Dir}})$,
	$$
	\int_{\Gw'}\!\!\left[\! A\nabla u_{\Upsilon}\!\cdot\! \nabla \phi+u_{\Upsilon}\bt\!\cdot\! \nabla \phi
	+\bb \!\cdot\!\nabla u_{\Upsilon} \phi\!+\!cu_{\Upsilon}\phi \right]\!\!\!\dx+\!\int_{\partial \Gw'_{\mathrm{Rob}}}\!\!\frac{\gamma}{\beta}u_{\Upsilon}\phi \dsigma\!=\!\!\int_{\Gw'}\!( g_0 \phi - \mathbf{g}\cdot \nabla\phi) \!\dx.
	$$	
\end{lemma}
 \begin{rem}\label{rem:certain_imp}
 	\em{ 
 		\begin{enumerate}
 			\item  If $(P,B)$ is symmetric, $\pwr\in C^{1,\alpha}$, $A\in C^{0,1}_{\loc}(\wb,\R^{n^2})$,  
 			$\bb=\bt\in C^{\alpha}_{\loc}(\wb,\R^n)$, and $c\in L^{\infty}_{\loc}(\wb)$, then any positive solution of the equation $(P,B)u=0$ in $\Gw$ belongs to $C^{1,\alpha}_{\loc}(\wb)$ \cite[Theorem 5.54]{L}.\\[1mm]
 			\item Assume that $(P,B)\mathbf{1}=0$ in $\Gw$, $A\in C^{0,1}_{\loc}(\wb,\R^{n^2})$, $\bt\in C^{0,1}_{\loc}(\wb,\R^n)$, $\bb\in L^{\infty}_{\loc}(\wb,\R^n)$ and $c\in L^{\infty}_{\loc}(\wb)$. Then, any  solution  of the equation $(P,B)u=0$ in $\Gw$ belongs to $W^{2,2}_{\loc}(\wb)$ (see  \cite[Theorems 5.29 and 5.54]{L}, where $\pwr\in C^2$ is assumed.). In fact,
 			using the `even' extension argument as in \cite[proof of Lemma 5.9 ]{PV2}, one obtains the above regularity result for the case $\pwr\in C^1$  (cf. \cite[Theorem 8.8]{GT}).
 		\end{enumerate}	
 	}
 \end{rem}
\subsection{Criticality theory}
The results in this subsection were recently proved in \cite{PV2}.
\begin{defi}
	\em{
Let $\Gw$ be a domain in $\R^n$, and let $0\lneqq V\in L^{p/2}_\loc(\wb)$, where $p>n$. The {\em generalized principal eigenvalue} of $(P,B)$ in $\Gw$ with respect to $V$  is defined by 
$$
\lambda_0=\lambda_0(P,B,V,\Gw):=\sup\{ \lambda\in \R :(P-\lambda V,B) \geq 0 \mbox{ in } \Gw \}.
$$
}
\end{defi}
\begin{defi}{\em 
		We say that the {\em generalized maximum principle} holds in a bounded domain $\Gw$ if for any $u\in H^1(\Gw)$ satisfying 
		$(P,B)u\geq 0$ in $\Gw$ and $u^{-}\in H^1_{\pwd}(\Gw)$,  we have
		$u \geq 0$ in $\Gw$.
	}
\end{defi}
 Assume that $(P,B)\geq 0$ in a domain $\Gw$. We recall the notion of  the ground state transform which implies a generalized maximum principle. Set 
	\begin{equation*}
	 \mathcal{RSH}_{P,B}(\Omega):= \{u\in \mathcal{SH}_{P,B}(\Omega)\mid u, u^{-1}, uPu \in L^\infty_\loc(\wb) \mbox{ and }  \frac{uBu}{\gb}\in  L^\infty_\loc(\Pw_{\mathrm{Rob}})  \} \neq\emptyset.
	\end{equation*}
 A supersolution $u\in \mathcal{RSH}_{P,B}(\Omega)$ is called a {\em regular positive supersolution of $(P,B)$ in $\Gw$}. 

\begin{defi}[Ground state transform]\label{gs+transform} {\em  Assume that $(P,B)\geq 0$ in a domain $\Gw$, and let $u\in \mathcal{RSH}_{P,B}(\Omega)$.
		Consider the   bilinear  form 
		$$
		\mathcal B_{P^u,B^u}(\phi,\psi):=\mathcal B_{P,B}(u\phi,u\psi)
		$$
		where  $ u\phi , u\psi \in \mathcal{D}(\Omega,\Pw_{\mathrm{Dir}})$.
		The form $\mathcal B_{P^u,B^u}$ corresponds to the elliptic operator 
		$$
		P^u(w):=\frac{1}{u}P(uw),\qquad  \mbox{with the boundary operator } B^u(w):=\frac{1}{u}B(uw).
		$$
		The operator $P^u$ is called the {\em ground state transform} of $P$ with respect to $u$. The operators  $P^u(w)$ and $B^u(w)$ are given explicitly by 
		$$
		P^u(w)=-\frac{1}{u^2}\diver(u^2A(x)\nabla w)+\left [		\bb-\bt\right ]\nabla w+\frac{Pu}{u}w , \mbox{ and }
		B^u(w)=\beta \langle A\nabla w,\vec{n}\rangle+ \frac{Bu}{u}w.
		$$
		
		We say that $w$ is a {\em weak (resp., super)solution of} $(P^u,B^u)$ in $\Omega$,  if  $w\in H^{1}_{\loc}(\wb)$  
		and for any (resp., nonnegative) $\phi$ such that $\phi u \in  \mathcal{D}(\Omega,\Pw_{\mathrm{Dir}})$ we have 
		\begin{multline}\label{eq:gs_eq}
		\mathcal B_{P^u,B^u}(w,\phi)\!=\!\int_{\Omega}\!\!\big(a^{ij}D_iw D_j\phi + (\bb^i - \bt^i)D_iw\phi \big) u^2\!\dx +\! \int_{\Omega}\!\! u(Pu) w\phi\dx  + 
		\!\int_{\Pw_{\mathrm{Rob}}}\!\!\!\!\! \frac{u Bu}{\gb} w \phi  \!\dsigma \\[2mm]
		= 0 \;(\text{resp., } \geq 0). 
		\end{multline}
	}
\end{defi}
\begin{rem}\label{rem:adj_Pu}
	\em{
		Note that  $\mathcal{B}_{P^u,B^u}$ is defined on   $L^2(\Gw, u^2 \!\dx)$.  Furthermore,
		if $u\in \mathcal{H}^{0}_{P,B}(\Gw)$, then the adjoint operator of $(P^u,B^u)$  in $L^2(\Gw,u^2\dx)$ is given by 
		$$ (P^u)^*w=\frac{P^*(uw)}{u}\,, \quad 
		\mbox{ and } \quad  (B^u)^*w=\frac{B^*(uw)}{u}\,.$$ 
	}
\end{rem}

\begin{rem}\label{r2.11}
	\em{
		\begin{enumerate}
			\item 	Let  $u,v\in \mathcal{H}_{P,B}^{0}(\Omega)$ (resp., $u\in \mathcal{H}_{P,B}^{0}(\Omega),v\in \mathcal{SH}_{P,B}(\Omega)$). Assume that $v/u \in H^{1}_{\mathrm{loc}}(\overline{\Omega}\setminus \pwd, u^2 \!\dx)$.
			Then $v/u$ is a weak positive (resp., super)solution of the equation $(P^u,B^u)w=0$ in $\Gw$.
			\item  If $u\in \mathcal{H}^{0}_{P,B}(\Omega)$, then for  any $w \in H^{1}_{\loc}(\overline{\Omega}\setminus \pwd)$ and $\phi \in  \mathcal{D}(\Omega,\Pw_{\mathrm{Dir}})$,  we have 
			$$
			\mathcal{B}_{P^u,B^u}(w,\phi)=\int_{\Omega}\big(a^{ij}D_iw D_j\phi + (\bb^i - \bt^i)D_iw\phi \big)u^2\dx. 
			$$  
		\end{enumerate}		
	}
\end{rem}
The following generalized maximum principle for a nonnegative operators \eqref{P,B} is proved in \cite[Lemma 3.19]{PV2} and is a consequence of the ground state transform.
\begin{lemma}[Generalized maximum principle]\label{gen_max_weak}
	 Assume that the operator $(P,B)$  is nonnegative in  $\Gw$.
	Consider a Lipschitz bounded subdomain  $\Gw'\Subset_R \Gw$.  If  $v\in H^1(\Gw')$ is a  supersolution of the equation $(P,B)u=0$ in $\Gw'$ with $v^{-}\in H^1_{\partial \Gw'_{\mathrm{Dir}} }(\Gw')$, then $v$ is nonnegative in $\Gw'$. 
\end{lemma}
We proceed with Harnack convergence principle \cite[Lemma 3.27]{PV}.
\begin{lem}[Harnack convergence principle]\label{HCP}
	Suppose that Assumptions~\ref{assump2} hold in $\Gw$, and  let $\{\Gw_k\}_{k\in \N}$ be an exhaustion of $\wb$. Let $x_0\in \Gw_1$ be a fixed reference point. 
	For each $k\geq 1$, let $ u_k\in H^{1}_{\loc}(\Gw_k \cup\partial \Gw_{k, \pwr})$ be a positive solution of the problem
	\begin{equation}\label{PB_Gwk}
	\begin{cases}
	Pu=0 &  \text{in } \Gw_k, \\
	Bu = 0 &  \text{on } \partial \Gw_{k,\mathrm{Rob}},
	\end{cases}
	\end{equation}
	satisfying  $u_k(x_0)=1$. Then the sequence $\{u_k\}_{k\in \N}$ admits a subsequence converging locally uniformly in $\wb$ to a positive solution 
	$u\in   ^{0}_{P,B}(\Gw)$. 
\end{lem}
\begin{rem}
	\em{
Lemma \ref{HCP} holds once  $P$ is replaced by $P^k(u)=-\diver(A^k \nabla u+\bt_k u)+\bb_k \nabla u+c_ku$,  where $A_k\in L^{\infty}_{\loc}(\wb)$
is a sequence of symmetric and positive definite matrices converging to $A$ in $L^{\infty}_{\loc}(\wb)$; $\bt_k,\bb_k\in L^p_{\loc}(\wb)$, $\bt_k\to \bt$  and $\bb_k\to \bb$ in $L^p_{\loc}(\wb,\R^n)$;  and $c_k\in  L^{p/2}_{\loc}(\wb)$, $c_k\to c$ in $L^{p/2}_{\loc}(\wb)$.
}
\end{rem}
The following characterization of $\lambda_0$ was proved in \cite{PV2}. 
\begin{theorem}[{\cite[Theorem~4.1]{PV2}}]\label{eeqiuv_add}
	 The following assertions are equivalent:
	\begin{enumerate}
		\item  $ {\mathcal H}^0_{P,B}(\Omega)\neq \emptyset$, and in particular, $\lambda_0(P,B,1,\Gw)\geq 0$.
		\item   $\mathcal{RSH}_{P,B}(\Omega)\neq \emptyset$.
		\item $\lambda_0(P,B,1,\Gw')>0$ for any Lipschitz subdomain $\Gw'\Subset_R\Gw$. 
		\item $(P,B)$ satisfies the generalized maximum principle in
		any Lipschitz subdomain  $\Gw'\Subset_R \Gw$. 
	\end{enumerate}
\end{theorem}
\begin{lem}[{\cite[ Lemma~ 4.3]{PV2}}]\label{lem1.8}
		Suppose that $H^{0}_{P,B}(\Gw)\neq \emptyset$, and let $\Gw'\Subset_R \Gw$ be a bounded Lipschitz subdomain of $\Gw$. 
	Let $K\Subset \Gw'$ be a Lipschitz  subdomain. 
	Then for any nonzero nonnegative function $f\in C_0^\infty(\Gw'\setminus K)$ there exists a unique positive weak solution $u\in H^1_{\partial \Gw'_{\Dir}\cup \partial K_{\Dir}}(\Gw'\setminus K)$  to  the problem 
	\begin{equation}\label{eq: with K}
	\begin{cases}
	Pw=f & \Omega'\setminus K , \\
	Bw= 0  & \partial {\Gw'}_{\mathrm{Rob}},\\
	\mathrm{Trace}(w)=0& (\partial {\Gw'}\cup \partial K)\setminus  \partial {\Gw'}_{\mathrm{Rob}}.
	\end{cases}
	\end{equation}
\end{lem}
Next, we introduce the notion of positive solution of minimal growth for \eqref{P,B} (cf. \cite{Agmon,PS,PV}). In the sequel $\{ \Gw_k\}_{k\in \N}$ is an exhaustion of $\wb$.
\begin{defi}\label{def:minimalgrowth}
	\em{
		A function $u$ is said to be a {\em positive solution of $(P,B)$ of minimal growth in a neighborhood of infinity in $\Gw$}  if 
		$u\in \mathcal{H}^{0}_{P,B}(\Gw^*_j)$ for some $j \geq 1$ and for any $l>j$ and $v\in C(\Gw^*_l \cup \Pw_{l,\Dir})\cap \mathcal{SH}_{P,B}(\Gw^*_l)$,  $u \leq v$ on $\Pw_{l,\Dir}$  $\Rightarrow$ $ u \leq v$ on $\Gw^*_l$
}
\end{defi}
\begin{lem}[{\cite[ Lemma~4.5]{PV2}}]\label{lem:min_growth}
	Suppose that $H^{0}_{P,B}(\Gw)\neq \emptyset$. Then for any $x_0\in \Gw$ the equation $(P,B)u=0$ has (up to a multiplicative constant) a unique positive solution $v$ in $\Gw \setminus \{ x_0\}$ of minimal growth in a neighborhood of infinity in $\Gw$.
\end{lem}
\begin{lem}[{\cite[ Lemma~ 4.7]{PV2}}]\label{inview}
	Suppose that $H^{0}_{P,B}(\Gw)\neq \emptyset$. Then $(P,B)$ is critical in $\Gw$ if and only if there exists a unique $u\in \mathcal{RSH}_{P,B}(\Omega)$ (up to a multiplicative constant).
\end{lem}
\begin{lem}[{\cite[Section~5]{PV2}}]
	Suppose that $H^{0}_{P,B}(\Gw)\neq \emptyset$. Then $(P,B)$ is subcritical in $\Gw$ if and only if $(P,B)$ admits a  unique positive minimal  Green function $G^{\Gw}_{P,B}(x,y)\in H^1_{\loc}(\wb\setminus \{y\})$ in $\Gw$ satisfying (in the sense of distributions)
$(P,B)G^{\Gw}_{P,B}(\cdot,y)=\delta_{y}$ in $\Gw$, where 
$\delta_{y}$ is the Dirac measure supported in $\{y\}.$
Moreover, $(P,B)$ is subcritical in $\Gw$ if and only if $(P^*,B^*)$ is subcritical in $\Gw$. In such a case,
$G^{\Gw}_{P,B}(x,y)=G^{\Gw}_{P^*,B^*}(y,x)$ for all $(x,y)\in \Gw\times \Gw$ satisfying $x\neq y$.
\end{lem}
The following elementary result concerning Green potentials is a consequence of \cite[Section~5.2]{PV2}.
\begin{proposition}\label{prop:GphiR}
	Assume that $(P,B)$ is subcritical in $\Gw$, and
	let $0\lneqq \vgf\in \core$. Then
	\begin{enumerate}
		\item $0<G_{\vgf}\in H^{1}_{\loc}(\wb)\cap C^{\alpha}(\wb)$,
		\item $(P,B)G_{\vgf}=\vgf$ in $\Gw$,
		\item For any $x_0\in \Gw$, there exists $K\Subset_R \wb$ such that  $G_{\vgf}\asymp G_{P,B}^{\Gw}(x,x_0)$ in $\Gw \setminus K$.
	\end{enumerate}
\end{proposition}
\begin{defi}\label{def_gs}
	\em{
		A function $u\in \mathcal{H}^{0}_{P,B}(\Gw)$ is called an {\em (Agmon) ground state} of $(P,B)$ in $\Gw$  if $u$ has minimal growth in a neighborhood of infinity in $\Gw$. 
	}
\end{defi}
We conclude this section with the following lemma.
\begin{lemma}[{\cite[ Lemma~4.10]{PV2}}]\label{gs_implies_crit}
		Suppose that $H^{0}_{P,B}(\Gw)\neq \emptyset$.
	Then $(P,B)$ admits a ground state in $\Gw$ if and only if $(P,B)$ is critical in $\Gw$.  
\end{lemma}
\section{Optimal Hardy inequalities for mixed boundary value problems}\label{sec:optimal}
The aim of this section is to prove theorems~\ref{main_thm} and \ref{thm3} concerning the existence of  families of  optimal Hardy-weights (see Definition~\ref{rem_1.3}). The case $\pwr=\emptyset$ (i.e., Dirichlet-Hardy-weights), has been studied in \cite{DFP,PV}.
We remark that all the results in this paper include the Dirichlet case, $\pwr=\emptyset$. 


First, we prove a Khas'minski\u{i}-type criterion for the criticality of the mixed boundary value problem \eqref{P,B} (cf. \cite[Proposition~6.1]{DFP} and references therein).
\begin{lem}[Khas'minski\u{i}-type criterion]\label{lem:Khasminsky}
  Let $u_0\in \mathcal{H}^{0}_{P,B}(\Gw)$ and $u_1\in
  \mathcal{SH}_{P,B}(\Gw\setminus K)\cap C(\overline{\Gw}\setminus (\pwd \cup K))$, where $K\Subset_R \Gw$ is a Lipschitz subdomain. Assume that 
  \begin{equation}\label{eq:K_crit}
  \lim_{x\to \infty_{\Dir}}\frac{u_0(x)}{u_1(x)}=0.
  \end{equation} 
  Then $u_0$ is a  a ground state of $(P,B)w=0$ in $\Gw$, and therefore, $(P,B)$ is critical in $\Gw$. 
\end{lem} 
\begin{proof}
	We need to prove that $u_0$ has minimal growth at infinity. Let $\{\Gw_k\}_{k\in \N}$ be an exhaustion of 
	$\wb$ and a Lipschitz subdomain $K'\Subset_R \Gw$ such that $K\Subset_R K'\Subset_R \Gw_1$, and $\partial(\Gw\setminus K')\cap \pwr$ is  Lipschitz  (see Remark \ref{rem:boundness}).
	Fix $x_0\in K$ and let $G(x)\in \mathcal{H}^{0}_{P,B}(\Gw \setminus \{x_0\})$ having minimal growth in a neighborhood of infinity in $\Gw$.  
 Let $C>1$ be fixed such that 
 \begin{equation}\label{eq:u_less_CG}
 C^{-1} G(x)\leq u_0(x) \leq CG(x) \qquad  \mbox{for all~} x\in \overline{\Gw_1 \setminus  K'}.
	\end{equation}
	The minimal growth  of $G(x)$ implies
	\begin{equation}\label{eq:letting_epsilon}
	C^{-1}G\leq  u_0 \qquad \mbox{in~} \Gw\setminus K'.
	\end{equation}
	Furthermore, \eqref{eq:K_crit} implies that for any $\varepsilon>0$, the exists $k_{\varepsilon}$ such that  
	for any $k\geq k_{\varepsilon}$ 
	$$u_0\leq \varepsilon u_1\leq CG+\varepsilon u_1 \qquad  \mbox{on~} \partial{\Gw_{k}\cap \Gw}.$$ 
	Notice that $u_0$ and $CG+\varepsilon u_1$ belong to 
	$C(\overline{\Gw}\setminus(\pwd \cup K))\cap \mathcal{SH}(\Gw\setminus K)$.
	Consider the set $D_k:=\Gw_k \setminus K'$.
	 Then, Corollary \ref{rem:ahlors} and \eqref{eq:u_less_CG}  imply that  $(CG+\varepsilon u_1-u_0)^{-}\in H^1_{\partial D_{k,\Dir} }(D_k)$.
	 The generalized maximum principle (Lemma \ref{gen_max_weak}) in $D_k$ then implies
	 \begin{equation*}
	 u_0\leq CG+\varepsilon u_1 \qquad \mbox{in~} D_k.
	 \end{equation*}
	Letting $k\to \infty $, we obtain 
	 \begin{equation*}
	u_0\leq CG+\varepsilon u_1 \qquad \mbox{in~} \Gw\setminus K'.
	\end{equation*}
	Letting $\varepsilon\to 0$ and \eqref{eq:letting_epsilon}  imply, 
	$u_0\asymp G$ in $\Gw\setminus K'$, namely, $u_0$ is a ground state, and therefore, $(P,B)$ is critical in $\Gw$. 
\end{proof}
As a corollary of Lemma \ref{lem:Khasminsky} we obtain the criticality claim in Theorem~\ref{main_thm}. 
\begin{theorem}\label{prop:criticality}
	Let Assumptions~\ref{assump2} hold in a domain $\Omega \! \subset \! \mathbb{R}^n$, ($n \! \geq \! 2$). Let $(P,B)$ be a subcritical   operator  in $\Omega$ and let $G(x):=G_{P,B}^{\Omega}(x,x_0)$ be its  minimal positive Green function with singularity at $x_0\in \Gw$.   Let $G_{\varphi}$ be the corresponding Green potential with  density $0\lneqq \varphi \in \core$. Assume that there exists $u\in \mathcal{H}^{0}_{P,B}(\Omega)$  satisfying Ancona's condition
	\begin{equation}\label{eq:G/u_tends}
	\lim_{x \to \infty_{\Dir}} \dfrac{G(x)}{u(x)}=0.
	\end{equation}
	
	Let $$ 0 \leq   a \, \leq \; \dfrac{1}{\sup\limits_{\Gw}{(G_{\varphi} /u)}}\, ,\quad
	f_w(t):=\sqrt{2t-at^2}, \quad W:=\frac{P \left (uf_w \left ( G_{\varphi} /u\right) \right )}{uf_w \left ( G_{\varphi} /u\right)} \, .
 	$$ 
	Then $(P-W,B)$
	is critical in $\Gw$.
\end{theorem}
\begin{proof}
By Proposition \ref{prop:GphiR} and Remark \ref{rem:boundness2}, we may assume without loss of generality 
that $G_{\varphi}/u<1$ in $\Gw$. 
	 Hence, $f_w(G_\vgf/u)$ is well defined and $f_w'(G_\vgf/u)>0$ in $\Gw$ .  Let $w(t):=(2t-at^2)^{-2}$. It can be easily checked  that  the functions $f_w$ and  
	$$f_1(t):=f_w(t)\int\limits_{t}^{1}\dfrac{1}{f_w^2(s)}\ds$$
	are linearly independent solutions of the equation
$$-y''-wy=0 \qquad \mbox{in } \R_{+}  $$
 which is related to the Ermakov-Pinney equation $-y'' =\frac{1}{y^3}$ (see \cite{PV2}).  
Moreover, $f_1$   is positive for $t<1$, negative for $t>1$, and satisfies
	\begin{equation}\label{eq:3.3}
	\lim_{t\to 0}\dfrac{f_w(t)}{f_1(t)}=\lim_{t\to \infty}\dfrac{f_w(t)}{f_1(t)}=0.
	\end{equation}
	
	Consider the functions $h,v:\Gw\to \R$ given by 
	$$v(x):=u(x)f_w \left(\dfrac{G_{\varphi}(x)}{u(x)}\right),  \qquad 
	h(x):=
	u(x)f_1 \left(\dfrac{G_{\varphi}(x)}{u(x)}\right),
	$$
	and recall that $G_{\vgf}$ is a positive solution of the problem 
	$$
	\begin{cases}
	Pw=\vgf& \mbox{in~} \Gw, \\
	Bw=0 & \mbox{on~} \pwr.
	\end{cases}
	$$
	A direct calculation
	 (see (4.13) in \cite{DFP}) shows that
    if	
      $0<G_{\vgf}/u<1$ in $\Gw$, then   we have 
	\begin{align}\label{Puv}
	Pv=& P(uf_w(G_{\vgf}/u)) = -uf_w''(G_{\vgf}/u)|\nabla (G_{\vgf}/u)|_A^2+uf_w'(G_{\vgf}/u)P(G_{\vgf}) = \nonumber\\ 
	& uf_w(G_{\vgf}/u)  w(G_{\vgf}/u) |  \nabla (G_{\vgf}/u)|_A^2 + uf_w'(G_{\vgf}/u)P(G_{\vgf})=Wv\! \geq \!0 \quad \mbox{in } \Gw,
	\end{align}
	and
	\begin{align} \label{Puv2}
	Ph= & \;P(uf_1(G_{\vgf}/u)) =    -uf_1''(G_{\vgf}/u)|\nabla (G_{\vgf}/u)|_A^2+uf_1'(G_{\vgf}/u)P(G_{\vgf}) =\nonumber \\
	& uf_1(G_{\vgf}/u) w(G_{\vgf}/u) |\nabla (G_{\vgf}/u)|_A^2+ uf_1'(G_{\vgf}/u)P(G_{\vgf}) \! \geq \! 0 \quad \mbox{in } \Gw.
	\end{align}
	Moreover,
		\begin{align*}
	&
	\beta  \frac{f_w'(G_{\vgf}/u)}{u}\langle uA\nabla G_{\vgf}-G_{\vgf}A\nabla u  ,\vec{n}\rangle = \\&
	f_w'(G_{\vgf}/u)\left(\beta  \langle A\nabla G_{\vgf} + G_{\vgf}\bt ,\vec{n}\rangle+ \gg G_{\vgf} \right)= 
	f_w'(g/u)BG_{\vgf} =0 
	\end{align*}
on $\pwr$ in the weak sense.	Therefore,
	\begin{align*}
	&Bv=
	B(uf_w(G_{\vgf}/u))=
	\beta f_w(G_{\vgf}/u)\langle A \nabla u  ,\vec{n} \rangle+\\&
	\beta  \frac{f_w'(G_{\vgf}/u)}{u}\langle uA\nabla G_{\vgf}-G_{\vgf}A\nabla u  ,\vec{n}\rangle 
	+\beta \langle uf_w(G_{\vgf}/u)\bt,\vec{n}\rangle+
	\gamma  uf_w(G_{\vgf}/u)  =  \\ &
	f_w(G_{\vgf}/u)Bu =  0.
	\qquad \mbox{ on~} \pwr \mbox{~in the weak sense}.
	\end{align*}
	Similarly, one can verify that $B(h) =  0$ on $\pwr$ in the weak sense. 
	As a result,  we obtain that  $v\in \mathcal{H}_{P-W,B}^{0}(\Gw)$ and $ h\in \mathcal{SH}_{P-W,B}(\Gw)\cap C(\overline{\Gw}\setminus \pwd )$. 
	Moreover,  \eqref{eq:G/u_tends} and \eqref{eq:3.3} imply
	\begin{equation*}
	\lim_{x\to \infty_{\Dir}}\dfrac{v(x)}{h(x)}=\lim_{x\to \infty_{\Dir}}\dfrac{u(x)f_w\left(\dfrac{G_{\varphi}(x)}{u(x)}\right)}{u(x)f_1 \left(\dfrac{G_{\varphi}(x)}{u(x)}\right)}=0.
	\end{equation*}
	By Lemma \ref{lem:Khasminsky}, $v(x)$ is a  ground state of $(P-W,B)$ in $\Gw$, and therefore, $(P-W,B)$ is critical in $\Gw$.
\end{proof}

\begin{rem}
	\em{
 Sufficient conditions for the existence of a function $u$ satisfying Ancona's condition \eqref{eq:G/u_tends} are known in the case $\pwr=\emptyset$. Indeed, Ancona  proved in \cite{Ancona02} that if $P$ is symmetric (or more generally quasi-symmetric in the sense of Ancona) such a positive solution $u$ exists. Moreover, Ancona gave a counter example of a nonsymmetric operator \cite{Ancona02} that does not admit such a positive solution $u$. It seems that Ancona's approach of constructing such a solution $u$  applies also to our setting \cite{Ancona03}.
		  
			Clearly,  in the nonsymmetric case, the existence of such a $u$ is guaranteed if  $\Gw$ is a bounded Lipschitz domain,  the coefficients of $(P,B)$ are up to the boundary regular  enough, $(P,B)\mathbf{1}=0$ in $\Gw$, and  $\pwr$, $\pwd \neq \emptyset$ are both relatively open and closed  disjoint smooth bounded sets. 
}
\end{rem}
 Note that $uf_w(G_\varphi/ u)$, the ground state of $(P-W,B)$ in $\Gw$, is H\"older continuous in $\wb$, and $W$ in \eqref{eq:W_explicitly} belongs  to $L^1_{\loc}(\wb)$. The following lemma  guarantees  the  H\"older continuity of the ground state of $(P^*-W,B^*)$, the adjoint operator of $(P-W,B)$.
\begin{lem}\label{lem:sim_argument}
	Let $W\in L^1_{\loc}(\wb)$, and assume that
	$(P-W,B)$ is critical in $\Gw$ with a ground state   $\psi \in \mathcal{H}^{0}_{P-W,B}(\Gw)\cap C^{\alpha}(\wb)$,  and let $\psi^{*}\in \mathcal{H}_{P-W,B}^{0}(\Gw)$ be the ground state of $((P-W)^*,B^*)$ in $\Gw$.
	Then $\psi^*=g^* \psi$, where
	$g^*$ is the ground state of $(((P-W)^\psi)^*,(B^\psi)^*)$ in $\Gw$. In particular, $\psi^*\in C^{\alpha}(\wb)$.
\end{lem}
\begin{proof}
	 Since $g^*$ is a weak solution of an operator satisfying Assumptions~\ref{assump2}, $g^*\in C^{\alpha}(\wb)$. Consequently, Remark \ref{rem:adj_Pu} implies that $\psi^*=g^* \psi\in C^{\alpha}(\wb)$.
\end{proof}

We continue with  the following elementary proposition which is needed later in the paper.

	\begin{proposition}{\cite[Proposition~6.3]{PV}}\label{SL_properties}	
	Let $w(t)=(2t-at^2)^{-2}$. Then for any $\xi,M,a>0$, and $t\in (0,2/a )$ the function 
	\be\label{eq_uxi}
	u_{\xi}(t):=\sqrt{2t-at^2}\cos\left(\dfrac{\xi}{2}\log \left (\dfrac{Mt}{2-at} \right )\right)
	\ee
	 satisfies the following properties:
	\begin{enumerate}
		\item $ -u_{\xi}''-(1+\xi^2)wu_{\xi}=0 \quad \mbox{in } (2/(M\mathrm{e}^{\pi/\xi}+a), 2/(M+a))$,\\ 
		\item The oblique boundary condition: $u_\xi ' (2/(M+a))=\dfrac{M^2-a^2}{4M}u_\xi (2/(M+a))$,\\
		\item The Dirichlet boundary condition: $u_{\xi}\! \left(2/(M\mathrm{e}^{\pi/\xi}+a)\right)=
		u_{3\xi}\! \left(2/(M\mathrm{e}^{\pi/\xi}+a)\right)=0$,\\
		\item $u_{\xi}$ converges pointwise to $\sqrt{2t-at^2}$ as $\xi \to 0$, \\
		\item $| u_{\xi}(t)| \leq \sqrt{2t-at^2}.$
	\end{enumerate}
\end{proposition}
%

%
\begin{thm}\label{3.7} 
	Assume that $(P,B)$, $G$, $G_{\varphi}$, $W$, and $u$ satisfy the assumptions of Theorem~\ref{prop:criticality}.  
	Assume further that  one of the following conditions is satisfied.
	\begin{enumerate}
		 
		\item  $(P,B)$ is symmetric,  $A\in C^{0,1}_{\loc}(\wb,\R^{n^2})$, $\bb=\bt\in C^{\alpha}_{\loc}(\wb,\R^n)$,  $c\in L^{\infty}_{\loc}(\wb)$, and $\pwr\in C^{1,\alpha}$.
		\\
		\item  $\pwr$, $\pwd$ are both relatively open and closed  sets, $\pwr$ is bounded and admits a finite number of connected components; the coefficients of $P$ are smooth functions in $\Gw$ (or more generally $A\in C^{\lceil(3n-2)/2\rceil,1}_{\loc}(\Gw,\R^{n^2})$, $\bt\in C^{\lceil(3n-2)/2\rceil,1}_{\loc}(\Gw,\R^{n})$, $\bb\in C^{\lceil(3n-2)/2\rceil-1,1}_{\loc}(\Gw,\R^{n})$, $c\in C^{\lceil(3n-2)/2\rceil-1,1}_{\loc}(\Gw)$).
	\end{enumerate}
	Then $(P-W,B)$ is  null-critical with respect to $W$.
\end{thm}
\begin{proof}

	 $(1)$ 
	 By Remark \ref{rem:certain_imp}, $u\in C^{1,\alpha}_{\loc}(\wb)$,
	 Therefore,   the ground state transformed operator, $(P^u,B^u)$ (Definition \ref{gs+transform}), satisfies the required regularity assumptions as well. 
	 As a result, we may assume without loss of generality  that $(P,B)\mathbf{1}=0$.
	Let $\xi>0$ be fixed and  consider the  set
	\begin{equation}\label{eq:Omega_xi}
	\Gw_{\xi}:=\left \{x\in \Gw \, \big|   \; \frac{2}{M\mathrm{e}^{\pi/\xi }+a }<G_{\varphi}(x)<\frac{2}{M+a}  \right \}. 
	\end{equation}
	Letting $M$ large enough, we may assume that
	 $$\mathrm{supp}(\vgf)\cap\Gw_{\xi} = \emptyset, \quad \mbox{and }  \Gw_{\xi}\Subset \wb\, .$$
	By Remark \ref{rem:certain_imp}, $G_{\vgf}\in C^{1,\alpha}_{\loc}(\wb)\cap W^{2,2}_{\loc}(\wb)$.
  $G_{\vgf}\in C^{1,\alpha}(\overline{\Gw_{\xi}})$ implies that $\Gw_{\xi}$ is a set of finite perimeter  for a.e $\xi$ \cite[proof of Theorem 5.9]{EG}. Therefore, we may use the coarea formula (cf. \cite[Lemma 9.2]{DFP},\cite[Theorem 5.9]{EG}) to obtain 
	\begin{align}\label{eq:nonumber}
		&
		\int_{\Gw_{\xi}}W\left( f_w(G_{\vgf})\right ) ^2 \dx=\nonumber
		\int_{\Gw_{\xi}}\frac{|\nabla G_{\vgf}|_A^2}{(2G_{\vgf}-aG_{\vgf}^2)}\dx=\\& 
		\int_{2/(M\mathrm{e}^{\pi/\xi}+a)}^{2/(M+a)}\frac{1}{2t-at^2}\dt\int_{G_{\vgf}=t}A\nabla G_{\vgf}\cdot\vec{n}\dsigma, 
	\end{align}
	 where for a.e. $t$, the vector  $\nabla G_{\vgf}(x)$ is parallel (in the metric $|\cdot|_A$) to  
	 the normal vector $\vec{n}(x)$ for $\mathcal{H}^{n-1}$- a.e. $x$ in the level set $\left\{G_{\vgf} = t \right\}$ \cite{CTZ}.
Moreover, $G_{\vgf}\in W^{2,2}_{\loc}(\wb)$ is a strong solution to the equation
$Pu=\vgf$ in $\Gw$  and satisfies $A \nabla G_{\vgf}\cdot \vec{n}=0$ everywhere on $\pwr$. Therefore,  we may use the divergence theorem for a.e. $t_1,t_2$, satisfying $$\frac{2}{M\mathrm{e}^{\pi/\xi}+a}<t_1<t_2<\frac{2}{M+a},$$ and obtain  
$$
0=-\int\limits_{\{x\in \Gw: \ t_1<G_{\vgf}<t_2\}}\diver(A\nabla G_{\vgf})\dx=\int\limits_{G_{\vgf}=t_2}A\nabla G_{\vgf}\cdot\vec{n}\dsigma-\int\limits_{G_{\vgf}=t_1}A\nabla G_{\vgf}\cdot\vec{n}\dsigma
$$
 (\cite[Proposition 3.1]{DP}).
 In particular, 
 $$
 \int_{G_{\vgf}=t_1}A\nabla G_{\vgf}\cdot\vec{n}\dsigma= \int_{G_{\vgf}=t_2}A\nabla G_{\vgf}\cdot\vec{n}\dsigma
 $$
 is a nonzero constant. Therefore, letting $\xi \to 0$ in \eqref{eq:nonumber} implies that 
 $$
\int_{\Gw}W\left (f_w(G_{\vgf})\right)^2 \dx\geq \lim_{\xi \to 0}\int_{\Gw_{\xi}}W\left(f_w(G_{\vgf})\right) ^2 \dx\asymp
 \int_{0}^{1}\frac{1}{2t-at^2}\dt=\infty.
 $$
 $(2)$ The proof in the non-symmetric case is identical to \cite[proof of Theorem 8.2]{DFP}. Indeed, our assumptions imply that  $\partial \Gw_{\xi}\cap \pwr=\emptyset$, where $\Gw_{\xi}$ is given by   \eqref{eq:Omega_xi}. Therefore, we may repeat almost literally the steps in  \cite[proof of Theorem 8.2]{DFP}.
 	\end{proof}
 \begin{proof}[Proof of Theorem~\ref{main_thm}]
 	The theorem follows from the criticality and the null-criticality of $(P-W,B)$ with respect to $W$ proved in Theorem~\ref{prop:criticality} and Theorem~\ref{3.7}, respectively.
 \end{proof}
%
%
\begin{rem}\label{rem:needed_for_barbatis}
	\em{ 1, 	The parameter  $a$ in the proof of Theorem \ref{main_thm} was chosen such that $G_{\varphi}/u<1/a$ in $\Gw$. In fact, we may choose any constant $a\geq 0$ satisfying $G_{\varphi}/u \leq 1/a$ in $\Gw$.
			
			2. In the nonsymmetric case of Theorem~\ref{3.7}, the stated smoothness assumptions on the coefficients simplify  the calculations in \cite[Theorem 8.2]{DFP}. In fact, these regularity assumptions can be  further weakened.
	} 
\end{rem}
The following proposition, a particular case of \cite[Proposition 3.1]{PV}, is a characterization of $\mathcal{W}(\R_{+})$, the set of all optimal Dirichlet-Hardy-weights of the Laplacian in $\R_+$.
\begin{proposition}\label{1D if and only if}
	Let $0\lneqq w\in L^1_{\loc}(\R_+)$. Then 	$w\in \mathcal{W}(\R_{+})$ with a corresponding ground state $\psi_w$  if and only if the following  three  conditions are satisfied.
	\begin{enumerate}
		\item $\psi_w >0$ satisfies $-\psi_w''-w\psi_w=0$ in $\R_{+}$,
		
		\item $\displaystyle{\int_{0}^{1}\dfrac{1}{\psi_w^2}\dt = \int_{1}^{\infty}\dfrac{1}{\psi_w^2}\dt=\infty}$,
		
		\item $\displaystyle{\int_{0}^{1}  \psi_w^2 w\dt=\int_{1}^{\infty}\psi_w^2 w\dt=  \infty}$.
	\end{enumerate}
\end{proposition}
We are now in a position to prove Theorem~\ref{thm3}.
\begin{proof}[Proof of Theorem~\ref{thm3}]
	1. Can be verified as in the proof of Theorem \ref{prop:criticality}. 
	\\[1mm]
	2. Use word by word the proof of  Theorem \ref{prop:criticality}.
	\\[1mm]
	3. The proof is identical to \cite[Theorem 5.2]{PV},  where $\pwr=\emptyset$ is assumed.
	\\[1mm]
	4. It remains to prove that  in the symmetric case, $(P-W,B)$ is null-critical with respect to $W$.
	Without loss of generality we may assume that $(P,B)\mathbf{1}=0$ in $\Gw$.
	Take $\ga >0$ sufficiently small such that 
	$$
	\{x\in \Gw \mid 0<G_{\varphi}(x)<\ga \} \cap \supp \varphi=\emptyset.
	$$
	For any $0<\varepsilon<\ga $, the coarea formula \eqref{eq:nonumber} implies
	\begin{align*}
		& \int\limits_{\varepsilon<G_{\varphi}<\ga } (\psi_w(G_{\vgf}))^2 W \dx=C
		\int\limits_{\varepsilon}^{\ga}\psi_w^2(t)w(t)  \dt.
	\end{align*}
Recall that $w\in \mathcal{W}(\R_+)$ with a corresponding ground state
$\psi_w(t)$. Therefore, letting $\varepsilon\to 0$, and using part (3) of Proposition~\ref{1D if and only if}  we obtain
$$
\int_{\Gw}
 (\psi_w(G_{\vgf}))^2 W \dx \geq C \lim\limits_{\varepsilon\to 0}
\int\limits_{\varepsilon}^{\ga}\psi_w^2(t)w(t)  \dt=\infty. \qquad \qedhere
$$
\end{proof}

\section{examples}\label{sec:exmaples}
In this short section we illustrate  two examples for which  Theorem~\ref{main_thm} provides an optimal Hardy-weight.
\begin{example}\label{ex:1}
	\em{
		Let $n\geq 3$, and  either 
		$$\Gw=B_1^{+}(0), \; \pwr=\{x \in B_1(0) \mid x_n = 0\},  \mbox { or }\;  
		\Gw=\R^n_+, \;  
		\pwr=\{x\in \R^n \mid  x_n=0\}.$$
		Consider the operator  $Pu:=-\Delta u$ together with the boundary operator $Bu=\nabla u\cdot \vec{n}$ on $\pwr$.  
		Clearly,   $(P,B)$ is subcritical in $\Gw$, and $(P,B)\mathbf{1}=0$ in $\Gw$. Indeed, 		
		for $x\in \Gw$ let $\hat{x}=(x',-x_n)$, 
		then, for each $x,y\in \Gw$ with $x\neq y$,
		$$
		G_{P,B}^{\Gw}(x,y)=
		\begin{cases}
		G_{P}^{B_1(0)}(x,y)+G_{P}^{B_1(0)}(\hat{x},y) & \Gw=B_1^{+}(0) \\
		G_{P}^{\R^n}(x,y)+G_{P}^{\R^n}(\hat{x},y) & \Gw=\R^n_+,
		\end{cases}
		$$
		where $G_{P}^{B_1(0)}(x,y)$ (resp., $G_{P}^{\R^n}(x,y)$) is the Dirichlet-Green function of $P$ in $B_1(0)$ (resp., $\R^n$).
		Obviously, $\lim\limits _{x\to \infty_{\Dir}}G_{\vgf}(x,y)=0$, and
		hence, Theorem~\ref{main_thm} implies that the function $W=\dfrac{P(f_w(G_{\vgf}))}{f_w(G_{\vgf})}$ is an optimal weight for $(P,B)$ in $\Gw$.
		
		We note that in the case $\Gw=B_1^{+}(0)$, $W(x)\sim (2\cdot \mathrm{dist}(x,\pwd))^{-2}$ as $x\to \xi$, where  $\xi_n>0$ and $|\xi|= 1$ \cite[Lemma~3.2]{LP}. \\
		On the other hand, in the case $\Gw= \R^n_+$, $W$ is a  continuous function in $\Gw$ and $W(x)\sim  \frac{(n-2)^2}{4}|x|^{-2} $ as $x\to \infty$  such that $x/|x|\to (\xi', \xi_n)$ with 
		$\xi_n>0$.
		
		
	}
\end{example}

\begin{example}\label{example:2}
	\em{
		Let $n\geq 3$, and  $\Gw=\{x\in \R^n \mid  |x|>1\}$ with $\pwr=\partial \Gw$.
		Assume that $Pu=-\Delta u$ and $Bu=\nabla u \cdot \vec{n}+\gamma(x) u$ on $\pwr$, where $ \gamma\in L^{\infty}(\pwr)$ satisfies $\gg > (1-n)/2$, and take $\varepsilon>0$ such that
		$\varepsilon(n+2\gamma-1)\geq 1$ on $\pwr$.
		Then, $$v:=\sqrt{(|x|-1+\varepsilon)|x|^{1-n}}\in \mathcal{RSH}_{P,B}(\Gw)$$ and
		$$
		\begin{cases}
		-\Delta v-\dfrac{(n-1)(n-3)v}{4|x|^2}-\dfrac{v}{4(|x|-1+\varepsilon)^2}=0 & \mbox{in~} \Gw, \\[6mm]
		\nabla v \cdot\vec{n}+\gamma v=\dfrac{-1+\varepsilon(n+2\gamma-1)}{2\sqrt{\varepsilon}}\geq 0 & \mbox{on~} \pwr. 
		\end{cases}
		$$
		Hence, the {\bf AP} theorem \cite{PV2} implies  the Hardy-type inequality in  $H^1(\Gw)$
		\begin{equation}\label{eq:K_L_B}
		\int_{\Gw}|\nabla \phi|^2\dx+\int_{\pwr}\!\!\gamma\phi^2 \!\dsigma \geq \int_{\Gw} \left [\dfrac{(n-1)(n-3)}{4|x|^2}+\dfrac{1}{4(|x|-1+\varepsilon)^2} \right] \!\phi^2\!\dx.
		\end{equation}
			} 
	\end{example}

		
		\begin{rem}[Improved Hardy-inequality in the exterior of the unit ball]\label{rem:improved_Laptev}
			\em{
				Assume further  that $\gamma\geq 0$ is constant.
				In \cite[Theorem 5.1]{KL}, \eqref{eq:K_L_B} has been obtained for the case $\varepsilon=(2\gamma)^{-1}$.
				Obviously,  by letting 
				$\varepsilon=\varepsilon_{\gamma}:= (n-1+2\gamma)^{-1}$ in \eqref{eq:K_L_B} we obtain an improvement of the Hardy-type inequality in \cite[Theorem 5.1]{KL}. In particular, the function $$v_{\gamma}:=\sqrt{(|x|-1+\varepsilon_{\gamma})|x|^{1-n}}$$ is a positive solution of the equation 
				\begin{equation}\label{eq:v_gamma_gs}
				\begin{cases}
				-\Delta v-\dfrac{v(n-1)(n-3)}{4|x|^2}-\dfrac{v}{4(|x|-1+\varepsilon_{\gamma})^2}=0 & \mbox{in~} \Gw, \\
				\nabla v \cdot\vec{n}+\gamma v=0 & \mbox{on~} \pwr. 
				\end{cases}
				\end{equation}
				Furthermore, the function $w_{\gamma}:=v_{\gamma}\log(|x|-1+\varepsilon_{\gamma})$ is a positive solution of the equation 
				\begin{equation}\label{eq:min_ex}
				-\Delta v-\dfrac{v(n-1)(n-3)}{4|x|^2}-\dfrac{v}{4(|x|-1+\varepsilon_{\gamma})^2}=0
				\end{equation}
				in a neighborhood of $\infty$, satisfying 
				$$
				\lim_{|x|\to \infty}\frac{v_{\gamma}}{w_{\gamma}}=0.
				$$
				By Lemma \ref{lem:Khasminsky}, any positive solution $\phi$ of \eqref{eq:min_ex} having minimal growth at $\infty$ satisfies $\phi\asymp v_{\gamma}$ in a neighborhood of $\infty$. As a result, $v_{\gamma}$ is the ground state of \eqref{eq:v_gamma_gs}. Moreover, it is easy to see that \eqref{eq:v_gamma_gs} is also null-critical, implying that the function 
				$$
				W:=\dfrac{(n-1)(n-3)}{4|x|^2}+\dfrac{1}{4(|x|-1+\varepsilon_{\gamma})^2}
				$$
				is an optimal Hardy-weight of $(P,B)$ in $\Gw$.
For an explicit formula for $G_{P,B}^{\Gw}(x,y)$ in the cases $\gamma=0,(n-2)/2$, see \cite{Sa}.				
			}
		\end{rem}

\begin{center}
{\bf Acknowledgments}
\end{center}
The paper is based on part of the Ph.~D. thesis of the second author under the supervision of the first author. I.~V.  is grateful to the Technion for supporting his study.
The  authors  acknowledge  the  support  of  the  Israel  Science Foundation (grant  637/19) founded by the Israel Academy of Sciences and Humanities. 

\end{document}